\documentclass[12pt]{amsart}
\usepackage{epsfig}
\usepackage{graphics}
\usepackage{dcpic, pictexwd}

\newtheorem{theorem}{Theorem}[section]
\newtheorem{Theorem}{Theorem}
\newtheorem{Corollary}[Theorem]{Corollary}
\newtheorem{lemma}[theorem]{Lemma}
\newtheorem{proposition}[theorem]{Proposition}

\newtheorem{question}[theorem]{Question}

\theoremstyle{definition}

\newtheorem{remark}[theorem]{Remark}

\theoremstyle{remark}

 \def\Z{{\mathbb{Z}}}

 \def\C{{\mathbb{C}}}
 
\def\mod{{\rm Mod}}

\def\GL{{\rm GL}}
\def\Aut{{\rm Aut}}
\def\Out{{\rm Out}}
\def\Inn{{\rm Inn}}
 \def\Sp{{\rm Sp}}
 \def\Im{{\rm Im}}
 \def\PSp{{\rm PSp}}

\begin{document}

\newenvironment{prooff}{\medskip \par \noindent {\it Proof}\ }{\hfill
$\square$ \medskip \par}
    \def\sqr#1#2{{\vcenter{\hrule height.#2pt
        \hbox{\vrule width.#2pt height#1pt \kern#1pt
            \vrule width.#2pt}\hrule height.#2pt}}}
    \def\square{\mathchoice\sqr67\sqr67\sqr{2.1}6\sqr{1.5}6}
\def\pf#1{\medskip \par \noindent {\it #1.}\ }
\def\endpf{\hfill $\square$ \medskip \par}
\def\demo#1{\medskip \par \noindent {\it #1.}\ }
\def\enddemo{\medskip \par}
\def\qed{~\hfill$\square$}

 \title[Linear Representations of mapping class group]
{Low-dimensional linear representations of mapping class groups}

\author[Mustafa Korkmaz]{Mustafa Korkmaz}

 \address{Department of Mathematics, Middle East Technical University,
  Ankara, Turkey, and}
 \address{Max-Planck Institut f\"ur Mathematik, Bonn, Germany}
 \email{korkmaz@metu.edu.tr}

 \date{\today}

\begin{abstract}
Recently, John Franks and Michael Handel proved that, for $g\geq 3$ and $n\leq 2g-4$,
every homomorphism from the mapping class group of an orientable surface of genus $g$
to $\GL (n,\C)$ is trivial. We extend this result to $n\leq
2g-1$, also covering the case $g=2$. As an application, we prove the
corresponding result for nonorientable surfaces. Another application is
on the triviality of homomorphisms from the
mapping class group of a closed surface of genus $g$ to $\Aut (F_n)$ or to $\Out (F_n)$
for  $n\leq 2g-1$.
\end{abstract}

 \maketitle

  \setcounter{secnumdepth}{1}
 \setcounter{section}{0}

\section{Introduction}
Let $S$ be a compact connected oriented surface of genus $g\geq 1$ with $q\geq 0$ boundary components and
with $p\geq 0$ marked points in the interior. Let $\mod (S)$ denote the mapping class group of $S$, the group of
isotopy classes of orientation--preserving self--diffeomorphisms of $S$. Diffeomorphisms and
isotopies are assumed to be the identity on the marked points and on the boundary.

There is the classical representation of $\mod (S)$ onto the symplectic group Sp$(2g,\Z)$
induced by the action of the mapping class group on the first homology of the closed surface
of genus $g$. In~\cite{fh}, Question~1.2,
Franks and Handel asked whether every homomorphism $\mod (S)\to \GL (n,\C)$ is trivial
for $n\leq 2g-1$. They proved that, in fact, this is the case for $g\geq 3$ and $n\leq 2g-4$,
improving a result of Funar~\cite{funar} who showed that every homomorphism from
the mapping class group to ${\rm SL} (n,\C)$ has finite image for $n\leq   \sqrt{g+1}$.

The aim of this paper is to give a complete answer to the question of Franks and Handel,
and then consider the corresponding problem for nonorientable surfaces.
We first prove the following theorem.

\begin{Theorem}  \label{thm:1}
  Let $g\geq 1$ and let $n\leq 2g-1$. Let $\phi:\mod (S)\to \GL (n,\C)$ be a homomorphism. Then
  the image ${\rm Im}(\phi)$ of $\phi$ is
  \begin{enumerate}
    \item    trivial if $g\geq 3$,
    \item   a quotient of the cyclic group $\Z_{10}$ of order $10$ if $g=2$, and
    \item   a quotient of $\Z_{12}$ if $(g,q)=(1,0)$ and of $\Z^{q}$ if $g=1$ and $q\geq 1$.
  \end{enumerate}
\end{Theorem}

I was informed by Bridson~\cite{b2} that he conjectured the above theorem. He hints this in the page $2$ of~\cite{b}.

The first corollary to Theorem~\ref{thm:1} is the following.

\begin{Corollary}
Let $g\geq 2$, $n\leq 2g-1$, $\Gamma$ be a quotient of $\mod (S)$
and let $\varphi:\Gamma\to \GL (n,\C)$ be a homomorphism. Then $\Im (\varphi)$
is trivial if $g\geq 3$, and  is isomorphic to a quotient of $\Z_{10}$ if $g=2$.
\end{Corollary}

Note that the groups $\Sp(2g,\Z)$, $\Sp(2g,\Z_m)$, $\PSp(2g,\Z)$, and $\PSp(2g,\Z_m)$ are quotients of $\mod (S)$.
Since $\mod (S)$ is residually finite~\cite{gros,iv1}, there are many finite quotients.

In the definition of the mapping class group, if we allow the diffeomorphisms of $S$ to permute the
marked points,
then we get a group $\mathcal{M} (S)$, which contains $\mod (S)$ as a subgroup of index $p!$.

\begin{Corollary}
  Let $g\geq 2$ and let $n\leq 2g-1$. Let $\phi:\mathcal{M} (S)\to \GL (n,\C)$ be a homomorphism. Then
  ${\rm Im}(\phi)$ is finite.
\end{Corollary}

For a nonorientable surface $N$ of genus $g$ with $p \geq 0$ marked points, we define the mapping class group
$\mod (N)$ of $N$ to be the group of isotopy classes of diffeomorphisms preserving the set of
marked points
(isotopies are assumed to fix the marked points). The action of a mapping class on the
first homology of the closed surface ${\overline N}$  obtained by forgetting the marked points
give rise to an automorphism of $H_1({\overline N};\Z)$ preserving the associated
$\Z_2$--valued intersection form.
It was proved by McCarthy and Pinkall~\cite{mp}, and also by Gadgil and Pancholi~\cite{gp}, that,
in fact, all automorphisms of $H_1({\overline N};\Z)$ preserving the $\Z_2$--valued intersection form
are induced by diffeomorphisms. By dividing out the torsion subgroup of $H_1 ({\overline N};\Z)$, we get a
representation $\mod (N)\to \GL (g-1,\C)$. It is now natural to ask the triviality of
the lower dimensional representations of $\mod (N)$.
Since the mapping class group $\mod (N)$ has nontrivial first homology, we cannot expect that every
such homomorphism is trivial. Instead, one may ask the following question.

\begin{question}
  Let $g\geq 3$, and let $n\leq g-2$. Is the image of every homomorphism
  $\phi:\mod (N)\to \GL (n,\C)$ finite?
\end{question}

As an application of Theorem~\ref{thm:1}, we answer this question leaving only one case open;
the case $g$ is even and $n=g-2$.

\begin{Theorem}  \label{thm:2}
  Let $g\geq 3$, and let $n\leq g-2$ if $g$ is odd and $n\leq g-3$ if $g$ is even. Let $N$ be a nonorientable
  surface of genus $g$ with $p \geq 0$ marked points and let $\phi:\mod (N)\to \GL (n,\C)$ be a homomorphism. Then
  ${\rm Im}(\phi)$ is finite.
\end{Theorem}

The mapping class group of a nonorientable surface with boundary components
may also be considered, but we restrict ourself to surfaces with marked points
only in order to make the proof simpler.

As another application of Theorem~\ref{thm:1}, we prove the following result
on the homomorphisms from the mapping class group of a closed orientable surface
to $\Aut (F_n)$ and to $\Out (F_n)$, where $F_n$ is the free group of rank $n$.
Compare Theorem~\ref{thm:autF-n} with Question~$16$ in~\cite{b-v}.

\begin{Theorem}  \label{thm:autF-n}
Let $g\geq 2$ and let $S$ be a closed orientable surface of genus $g$. Let $n$ be a positive integer with
$n\leq 2g-1$. Let $H$ denote either of $\Aut (F_n)$ or $\Out (F_n)$ and let
$\varphi:\mod (S)\to H$ be a homomorphism. Then the image of $\varphi$ is
  \begin{enumerate}
    \item    trivial if $g\geq 3$, and
    \item   a quotient of  $\Z_{10}$ if $g=2$.
\end{enumerate}
\end{Theorem}

In~\cite{fh}, the main theorem, Theorem~1.1, is proved by induction on $g$. It is first proved for the cases
$g\geq 3$ and $n\leq 2$. The main improvement of this paper is that we can start the induction
from the cases $g=2$ and $n\leq 3$. The rest of the proof of Theorem~\ref{thm:1} follows from
the arguments of~\cite{fh}. We give a slight modification of this proof.
The proofs in cases $g=2$ and $n\leq 2$ follow from the proof of Lemma~3.1
of~\cite{fh} with the additional information that the commutator subgroup of $\mod (S)$ is
perfect. In the case $(g,n)=(2,3)$ we need to treat all possible Jordan forms
of the image of the Dehn twist about a nonseparating simple closed curve.

\medskip

\noindent
{\bf Acknowledgments.}
This paper was written while I was visiting the Max-Planck Institut f\"ur Mathematik in Bonn.
I thank MPIM for its generous support and wonderful research environment.
After the completion of the first version of this work, I was informed by John Franks and Michael Handel that they also
improved Theorem~1.1 in~\cite{fh} to the cases $g\geq 2$ and $n \leq 2g-1$. I would like to thank them
for sending the new version of their paper, which now appears on Arxiv as version 3. I also thank Martin Bridson for his
interest in this work.


\section{Algebraic preliminaries}

We state two properties of subgroups of $\GL (n,\C)$. They are either well-known, or easy to prove.
Therefore, we do not prove them. These properties will be used in the proof of Theorem~\ref{thm:1}.

\begin{lemma}  \label{lem:matrix2}
Let $C= \left(
         \begin{array}{ccc}
             z & * & * \\
             0 & z & * \\
             0 & 0 & z \\
        \end{array} \right)$
 and $D=\left(
         \begin{array}{ccc}
             w & * & * \\
             0 & w & * \\
             0 & 0 & w \\
        \end{array} \right) $
be two elements of $\GL (3,\C)$. Then $CDC=DCD$ if and only if $C=D$.
\end{lemma}

\begin{lemma}  \label{lem:uptrglr}
The subgroup of $\GL (n,\C)$ consisting of upper triangular matrices is solvable.
\end{lemma}

We will also require the following lemma from~\cite{fh}.

\begin{lemma} \label{prop:solv} $($\cite{fh}, {\rm Lemma} $2.2)$
  Let $G$ be a perfect group, and $H$ a solvable group. Then any homomorphism $G\to H$ is trivial.
\end{lemma}


\section{Mapping class groups and commutator subgroups}

Let $S$ be a compact oriented surface of genus $g$ with $p\geq 0$ marked points
and with $q\geq 0$ boundary components. In this section we give the results on mapping class groups required in
the proof of Theorem~\ref{thm:1}. For further information on mapping class groups,
the reader is referred to~\cite{iv}, or~\cite{fm}. For a simple closed curve $a$ on $S$
we denote by $t_a$ the (isotopy class of the) right Dehn twist about $a$.

\begin{theorem}  \label{thm:dualeqv} $($\cite{km}, {\rm Theorem} $1.2)$
 Let $g\geq 1$. Suppose that $a$ and
$b$ are two nonseparating simple closed curves on $S$. Then there is a sequence
\[
a=a_0,a_1,a_2,\ldots,a_k=b
\]
of nonseparating simple closed curves such that $a_{i-1}$ intersects $a_i$ at only one point
\end{theorem}

\begin{theorem}
  Let  $g\geq 2$. Then the mapping class group $\mod (S)$ is generated Dehn twists about nonseparating
  simple closed curves on $S$.
\end{theorem}

\begin{theorem}  \label{thm:G'} $($\cite{km}, {\rm Theorem} $2.7)$
  Let  $g\geq 2$. Let $a$ and $b$ be two nonseparating simple closed
  curves on $S$ intersecting at one point. Then the commutator subgroup of $\mod (S)$ is generated normally by $t_at_b^{-1}$.
\end{theorem}

Theorem~\ref{thm:G'} should be interpreted as follows: If a normal subgroup of $\mod (S)$ contains
$t_at_b^{-1}$, then it contains the commutator subgroup of $\mod (S)$, the (normal) subgroup
generated by all commutators $[x,y]=xyx^{-1}y^{-1}$.

\begin{theorem}  \label{thm:G'prf} $($\cite{km}, {\rm Theorem} $4.2)$
  Let $g\geq 2$. Then the commutator subgroup of
  $\mod (S)$ is perfect.
\end{theorem}

Note that, in~\cite{km}, the group $\mod (S)$ of this paper is denoted by ${\mathcal{P}M}_S$. In that paper,
Theorems~\ref{thm:G'} and~\ref{thm:G'prf} above are proved for surfaces with marked points (=punctures),
but the same proof apply to
surfaces with boundary as well since $\mod (S)$ is generated by Dehn twists about nonseparating
simple closed curves.

We record the following well--known relations among Dehn twists.

\begin{lemma}
  Let $a$ and $b$ be two simple closed curves on $S$, and let $t_a$ and $t_b$
  denote the right Dehn twists about them.
  \begin{enumerate}
    \item If $a$ and $b$ are disjoint, then $t_a$ and $t_b$ commute.
    \item If $a$ intersects $b$ transversely at one point, then they satisfy the
    braid relation $t_at_bt_a=t_bt_at_b$.
  \end{enumerate}
\end{lemma}

Recall that the first homology group $H_1(G;\Z)$ of a group $G$ is isomorphic to
the abelianization $G/G'$, where $G'$ is the (normal) subgroup of
$G$ generated by all commutators $[g_1,g_2]$.

\begin{theorem}  \label{thm:H_1} $($\cite{k02tjm}, {\rm Theorem} $5.1)$
  Let $g\geq 1$. Then the first homology group
  $H_1(\mod (S);\Z)$ is
  \begin{itemize}
    \item[(1)] trivial if $g\geq 3$,
    \item[(2)]  isomorphic to the cyclic group of order $10$ if $g=2$,
    \item[(3)]  isomorphic to the cyclic group of order $12$ if $(g,q)=(1,0)$, and
    \item[(4)]  isomorphic to $\Z ^q$ if $g=1$ and $q\geq 1$.
  \end{itemize}
\end{theorem}

Note that since any two Dehn twists about nonseparating simple closed curves are conjugate in $\mod (S)$,
their classes in $H_1(\mod (S);\Z)$ are equal. In particular, we conclude the next lemma.

\begin{lemma}  \label{lem:abelianrep} Let $g\geq 1$, and let $b$ and $c$ be two nonseparating simple closed curves on $S$.
If $H$ is an abelian group and if $\phi :\mod (S)\to H$ is a homomorphism, then $\phi (t_b)= \phi (t_c)$.
\end{lemma}


\section{Homomorphisms $\mod(S)\to \GL(n,\C)$}
In this section we prove Theorem~\ref{thm:1}. So let $n\leq 2g-1$ and let $\phi:\mod(S)\to \GL(n,\C)$
be a homomorphism.

For the proof, we adopt the proof of Theorem~1.1 in~\cite{fh}.
Franks and Handel use the fact that when $g\geq 3$ the group $\mod (S)$
is perfect. If this is rephrased as "the commutator subgroup of $\mod (S)$ is perfect" then
it is still true for the case $g=2$ as well, and that is what we use below.
The proof given in~\cite{fh} for the case $g\geq 3$ and $n\leq 2$,
also works for the case $g=2$ and $n\leq 2$ with a slight modification. For the
case $g=2$ and $n=3$, we need to analyze six possible Jordan forms of the image of the
Dehn twist about a nonseparating simple closed curve. Once Theorem~\ref{thm:1} is established for
the cases $g=2$ and $n\leq 3$, we then again follow a modified version of the idea of Franks and Handel to
induct $g$.

Following~\cite{fh}, for a simple closed curve $x$ on $S$ we denote $\phi (t_x)$ by $L_x$. If $\lambda$
is an eigenvalue of a linear operator $L$, the corresponding eigenspace is denoted by $E_\lambda (L)$.
We write $E_\lambda^x$ for $E_\lambda (L_x)$.

\begin{proposition} \label{prop:2}
  Let $g= 2$ and $n\leq 2$. Then ${\rm Im}(\phi)$ is a quotient of the cyclic group $\Z_{10}$.
  \end{proposition}
\begin{proof}
If $n=1$ then $\GL (n,\C)=\C^*$ is abelian. Hence, $\phi$ factors through the first homology of
$\mod (S)$, which is isomorphic to $\Z_{10}$. So assume that $n=2$. In this proof, we set $G=\mod (S)$.

Let $a,b$ and $c$ be three nonseparating simple closed curves on $S$
such that $a$ is disjoint from $b\cup c$,
and that $b$ intersects $c$ transversely at one point.
 Clearly, in order to complete the proof, it suffices to prove that $\phi (G')=\{I\}$,
where $G'$ is the commutator subgroup of $G$. There are three possibilities for the Jordan form of $L_a$.

(i). Suppose that $L_a$ has two distinct eigenvalues $\lambda_1$ and $\lambda_2$, with corresponding
eigenvectors $v_1$ and $v_2$; $L_a(v_i)=\lambda_i v_i$. With respect to the basis $\{ v_1, v_2\}$,
the matrix $L_a$ is diagonal. Since $L_b$ and $L_c$ preserve each eigenspace of $L_a$, they are diagonal too.
In particular, they commute. Now from the braid relation $L_bL_cL_b=L_cL_bL_c$, we get $L_b=L_c$, or
$\phi(t_bt_c^{-1})=I$. Since $G'$ is generated normally by $t_bt_c^{-1}$ (c.f. Theorem~\ref{thm:G'}),
we conclude that $\phi (G')$  is trivial

(ii). If the matrix $L_a$ has only one eigenvalue $\lambda$ and if the Jordan form of $L_a$ is $\lambda I$,
then $L_x=\lambda I$ for each nonseparating simple closed curve $x$ on $S$. This is because $L_x$ is conjugate
to $L_a$. Since the group
$G$ is generated by Dehn twists about such curves,
we have that $\phi (G)$ is cyclic.

(iii).
Suppose finally that the Jordan form of $L_a$ is not diagonal, so that the matrix of $L_a$ is
\[ \left(
  \begin{array}{cc}
    \lambda & 1 \\
    0 & \lambda \\
  \end{array}
\right)
\]
in some fixed basis. Because $L_b$ and $L_c$ preserve the eigenspace of $L_a$, with respect to the same basis,
the matrices $L_b$ and $L_c$ are upper triangular whose diagonal entries are $\lambda$. In particular,
we have $L_bL_c=L_cL_b$. From the braid relation $L_bL_cL_b=L_cL_bL_c$ again, we get $L_b=L_c$, or
$\phi(t_bt_c^{-1})=I$. From this we conclude that $\phi (G')$  is trivial.

This completes the proof of the proposition.
\end{proof}

\begin{lemma} \label{lem:eigenspace}
Let $a,b,x,y$ be four nonseparating simple closed curves on $S$ such that there is
an orientation--preserving diffeomorphism $f$ of $S$ with $f(x)=a$ and $f(y)=b$.
Let $\lambda$ be an eigenvalue of $L_a=\phi (t_a)$.
Then $E_\lambda^a=E_\lambda^b$ if and only if $E_\lambda^x=E_\lambda^y$.
\end{lemma}
\begin{proof}
Let $F=\phi (f)$.
The assumptions $f(x)=a$ and $f(y)=b$ imply that $ft_xf^{-1}=t_a$ and $ft_yf^{-1}=t_b$, and hence
$FL_xF^{-1}=L_a$ and $FL_yF^{-1}=L_b$. Therefore,
\[ E_\lambda^a = E_\lambda (L_a) =E_\lambda (FL_xF^{-1})=F (E_\lambda^x)\] and
\[  E_\lambda^b = E_\lambda (L_b) =E_\lambda (FL_yF^{-1})=F (E_\lambda^y).
 \]

 The lemma now follows from these two.
\end{proof}

\begin{proposition} \label{prop:g=2}
  If $g= 2$ then the image of any homomorphism  $\phi:\mod (S)\to \GL (3,\C)$
  is a quotient of the cyclic group $\Z_{10}$
\end{proposition}
\begin{proof}
Let $G=\mod (S)$, and let $G'$ denote the commutator subgroup of $G$.
Since $H_1(G,\Z)$ is isomorphic to $\Z_{10}$,
it suffices to prove that $\phi (G)$ is abelian, or equivalently that $\phi (G')$ is trivial.

Since $\phi (G')$ is generated normally by a single element $t_xt_y^{-1}$ for any two
nonseparating simple closed curves $x$ and $y$ intersecting at one point,
it is also sufficient to find two such curves with $\phi ( t_xt_y^{-1})=I$.

Let $a$ be a nonseparating simple closed curve on $S$.
The Jordan form of $L_a$ is one of the following six matrices:\\

  (i)
 $ \left(
   \begin{array}{ccc}
     \lambda_1 & 0 & 0 \\
     0 & \lambda_2 & 0 \\
     0 & 0 & \lambda_3 \\
   \end{array}
 \right),$ \ (ii)
 $ \left(
   \begin{array}{ccc}
     \lambda & 0 & 0 \\
     0 & \lambda & 0 \\
     0 & 0 & \lambda \\
   \end{array}
 \right),$  \ (iii)
 $ \left(
   \begin{array}{ccc}
     \lambda & 0 & 0 \\
     0 & \mu & 1 \\
     0 & 0 & \mu \\
   \end{array}
 \right),$

 (iv)
 $ \left(
   \begin{array}{ccc}
     \lambda & 1 & 0 \\
     0 & \lambda & 1 \\
     0 & 0 & \lambda \\
   \end{array}
 \right),$ \  (v)
 $ \left(
   \begin{array}{ccc}
     \lambda & 0 & 0 \\
     0 & \lambda & 1 \\
     0 & 0 & \lambda \\
   \end{array}
 \right),$  \ (vi)
  $ \left(
   \begin{array}{ccc}
     \lambda & 0 & 0 \\
     0 & \lambda & 0 \\
     0 & 0 & \mu \\
   \end{array}
 \right) .$\\
Here, distinct notations represent distinct eigenvalues.
In each case we fix a basis with respect to which the matrix $L_a$ is in its Jordan form.
Recall that the eigenspace of $L_x$ corresponding to an eigenvalue $\lambda$ is denoted by $E_\lambda^x$.
We now analyze each case.

(i). In this case, each eigenspace of $L_a$ is $1$--dimensional. Let $b_1$ and $b_2$ be two
nonseparating simple closed curves intersecting at one point such that each $b_i$ is disjoint from $a$.
Then each eigenspace $E_{\lambda_i}^a$ is invariant under each $L_{b_i}$,
so that $L_{b_i}$ are diagonal. In particular, $L_{b_1}$ and $L_{b_2}$ commute. Now the braid relation
$L_{b_1}L_{b_2}L_{b_1}=L_{b_2}L_{b_1}L_{b_2}$ implies that $L_{b_1}=L_{b_2}$. Thus, we have $\phi (L_{b_1})= \phi (L_{b_2})$,
i.e. $\phi (t_{b_1}t_{b_2}^{-1})=I$.

(ii). If $b$ is a nonseparating simple closed curve on $S$, then $L_b$ is conjugate to
$L_a$, so that $L_b=\lambda I$. Since $G$ is generated by all such Dehn twists,
we get that $\phi (G)$ is cyclic.

(iii). Let $b_1$ and $b_2$ be two nonseparating simple closed curves intersecting at one point
such that they are disjoint from $a$.
The matrices $L_{b_1}$ and $L_{b_2}$ preserve each eigenspace $E_\lambda^a$ and $E_\mu^a$, so that they are
of the form
\[ L_{b_i}=\left(
   \begin{array}{ccc}
     x_i & 0 & w_i \\
     0 & y_i & u_i \\
     0 & 0  & z_i \\
   \end{array}
 \right).\]
The braid relation $L_{b_1}L_{b_2}L_{b_1}=L_{b_2}L_{b_1}L_{b_2}$ implies that $x=x_1=x_2$, $y=y_1=y_2$ and $z=z_1=z_2$.
The equality $L_{b_i}L_a=L_aL_{b_i}$ then gives $w_i=0$ and $y=z$. Hence, we have $x=\lambda$, $y=z=\mu$, and so
\[ L_{b_i}=\left(
   \begin{array}{ccc}
     \lambda & 0 & 0 \\
     0 & \mu & u_i \\
     0 & 0  & \mu \\
   \end{array}
 \right).\]
But then we have $L_{b_1}L_{b_2}=L_{b_2}L_{b_1}$. The braid relation
$L_{b_1}L_{b_2}L_{b_1}=L_{b_2}L_{b_1}L_{b_2}$ again implies that $L_{b_1}=L_{b_2}$, and hence
 $\phi ( t_{b_1}t_{b_2}^{-1})=I$.

(iv). Let $b_1$ and $b_2$ be two nonseparating simple closed curves intersecting at one point
such that they are disjoint from $a$.
In this case, $\ker (L_a-\lambda I)=E_{\lambda}^a$ is $1$-dimensional, $\ker (L_a-\lambda I)^2$ is $2$-dimensional,
and they are $L_{b_i}$--invariant for $i=1,2$. It follows that
\[ L_{b_i}=\left(
   \begin{array}{ccc}
     \lambda & * & * \\
     0 & \lambda &  * \\
     0 & 0 & \lambda \\
   \end{array}
 \right).\]
Since there is the braid relation $L_{b_1}L_{b_2}L_{b_1}=L_{b_2}L_{b_1}L_{b_2}$, Lemma~\ref{lem:matrix2}
implies that $L_{b_1}=L_{b_2}$. Hence, $\phi ( t_{b_1}t_{b_2}^{-1})=I$.

(v).
The eigenspace $E_\lambda^a$ is $2$-dimensional in this case.
Suppose first that $E_\lambda^a \neq E_\lambda^b$ for some (hence all)
nonseparating simple closed curve $b$
intersecting $a$ at one point. Choose two nonseparating simple closed curves $c_1$ and $c_2$ disjoint from $a\cup b$
such that $c_1$ intersects $c_2$ at one point. Then $E_\lambda^a \cap E_\lambda^b$ and $E_\lambda^a$
are $L_{c_i}$--invariant subspaces, so that with respect to a suitable basis
\[L_{c_i}=\left(
   \begin{array}{ccc}
     \lambda & * & * \\
     0 & \lambda &  * \\
     0 & 0 & \lambda \\
   \end{array} \right) .\]
Note that we have the braid relation
$ L_{c_1} L_{c_2}L_{c_1}=L_{c_2}L_{c_1} L_{c_2}$.
We now use Lemma~\ref{lem:matrix2} to conclude that $L_{c_1}=L_{c_2}$, so that
$\phi (t_{c_1}t_{c_2}^{-1})=I$.

If $E_\lambda^a = E_\lambda^b$ for some (hence all) $b$ intersecting $a$ at one point, then $E_\lambda^a = E_\lambda^x$
for all nonseparating simple closed curves $x$. This is because, by Theorem~\ref{thm:dualeqv}
there is a sequence
$a=a_0,a_1,a_2,\ldots,a_k=x$ of nonseparating simple closed curves such that
$a_{i-1}$ intersects $a_i$ at one point for all $1\leq i\leq k$,
 and $E_\lambda^{a_{i-1}} = E_\lambda^{a_{i}} $. Since $G$ is generated by Dehn twists about nonseparating
 simple closed curves, $E_\lambda^a $ is $\phi (G)$--invariant,
so that $\phi$ induces a homomorphism $\bar\phi: G \to \GL (E_\lambda^a )=\GL (2,\C)$. By Proposition~\ref{prop:2},
$\bar \phi (G)$ is cyclic, and hence $\bar \phi (f)=I$ all $f\in G'$. Thus the matrix
of $\phi (f)$ is of the form
\[ \left(
   \begin{array}{ccc}
     1 & 0 & z_1 \\
     0 & 1 &  z_2 \\
     0 & 0 & z_3 \\
   \end{array}
 \right).\]
Since the subgroup of $\GL (3,\C)$ consisting of upper triangular matrices are solvable and since $G' $ is perfect,
$\phi (G')$ is trivial.

(vi).
In this last case, the eigenspace $E_\lambda^a$ is, again, $2$-dimensional.
If $E_\lambda^a = E_\lambda^b$ for some  nonseparating simple closed curve $b$ intersecting $a$
at one point, then from Theorem~\ref{thm:dualeqv} and Lemma~\ref{lem:eigenspace}
we obtain that $E_\lambda^a = E_\lambda^x$ for all nonseparating simple closed curves $x$.
We conclude now as in the case (v) that $\phi (G')$ is trivial.

Suppose finally that $E_\lambda^a \neq E_\lambda^b$ for some (hence all) nonseparating simple closed curve $b$
intersecting $a$ at one point. By Lemma~\ref{lem:eigenspace}, $E_\lambda^x \neq E_\lambda^y$
for all nonseparating simple closed curves $x$ and $y$ intersecting once.
Let $a=c_4$ and $b=c_5$. Choose three nonseparating simple closed curves $c_1,c_2,c_3$
such that
\begin{itemize}
  \item $c_i$ intersects $c_j$ at one point if $|i-j|=1$, and
  \item $c_i$ is disjoint from $c_j$ if $|i-j|\geq 2$.
\end{itemize}
Let $v_1\in E_\lambda^a \cap E_\lambda^b$, $v_2\in E_\lambda^a$ and $v_3\in E_\mu^a$
so that $\{v_1,v_2,v_3\}$ is a basis. With respect to
this basis, the matrix of $L_a$ is its Jordan matrix.
Since $E_\lambda^a \cap E_\lambda^b$, $E_\lambda^a$ and $E_\mu^a$
are $L_{c_i}$--invariant for $i=1,2$, we have
\[L_{c_i}=\left(
   \begin{array}{ccc}
     x_i & w_i & 0 \\
     0 & y_i &  0 \\
     0 & 0 & z_i \\
   \end{array} \right) ,\]
 with $\{x_i,y_i,z_i\}=\{\lambda,\mu \}$. Then the braid relation
\begin{equation} \label{eqn:braid1}
 L_{c_1} L_{c_2}L_{c_1}=L_{c_2}L_{c_1} L_{c_2}
\end{equation}
gives us $x_1=x_2=x, y_1=y_2=y$ and $z_1=z_2=z$.

If $z=\mu$ then $x=y=\lambda$, and hence $L_{c_1}L_{c_2}=L_{c_2}L_{c_1}$. Now the braid
relation~(\ref{eqn:braid1}) implies again that $L_{c_1}=L_{c_2}$. It follows that $\phi (G')$ is trivial as above.

Suppose that $z=\lambda$. We will show that this case is not possible by arriving at a contradiction.
If $x=\lambda$ then $y=\mu$. But then we have $E_{\lambda}^{c_1}=E_{\lambda}^{c_2}$, which is a contradiction.
If $x=\mu$ then $y=\lambda$. Since $E_{\mu}^{c_1}$ is
$L_{c_3}$--invariant, we have
\[L_{c_3}=\left(
   \begin{array}{ccc}
     u & * & * \\
     0 & * &  * \\
     0 & * & * \\
   \end{array} \right) .\]
Now the braid relation $L_{a} L_{c_3}L_{a}=L_{c_3}L_{a} L_{c_3}$ gives $u=\lambda$,
while the braid relation $L_{c_3} L_{c_2}L_{c_3}=L_{c_2}L_{c_3} L_{c_2}$ gives $u=\mu$.
Since $\lambda\neq \mu$, we get a contradiction again.

This completes the proof of the proposition.
\end{proof}

Finally, we are in a position to prove Theorem~\ref{thm:1}.

\subsubsection{Proof of Theorem~\ref{thm:1}}
If $g=1$ then $n=1$, and $\phi :\mod (S)\to \GL(1,\C)$. Since $\GL(1,\C)$ is abelian,
$\phi$ factors through $H_1 (\mod (S);\Z)$. The theorem now follows from Theorem~\ref{thm:H_1}.
If $g=2$ then $n\leq 3$. These cases are proved in Propositions~\ref{prop:2} and ~\ref{prop:g=2}.
We assume that $g\geq 3$ and that the theorem holds true
for all surfaces of genus $g-1$. Since $\GL(k-1,\C)$ is isomorphic to a subgroup of $\GL(k,\C)$,
it suffices to prove the theorem for $n=2g-1$.

In what follows $R$ denotes a subsurface of $S$ diffeomorphic to a compact connected
surface of genus $g-1$ with one boundary component. We embed $\mod (R)$ into $\mod (S)$ by extending
self-diffeomorphisms of $R$ to $S$ by the identity. We set $G=\mod (S)$ and $H_R=\mod (R)$.

If, for some subsurface $R$, there exists a $\phi (H_R)$--invariant subspace $V$ of dimension $r$ with
$2\leq r \leq n-2$, then $\phi$ induces homomorphisms
$\phi_1: H_R \to \GL(V)=\GL (r,\C)$ and $\phi_2: H_R \to \GL(\C^n/V)=\GL (n-r,\C)$.
Note that $r\leq 2(g-1)-1$ and $n-r\leq 2(g-1)-1$. By assumption, the image of each $\phi_i$ is cyclic.
In particular, if $b$ and $c$ are two simple closed curves on $R$
intersecting at one point, then $\phi_i(t_b)=\phi_i(t_c)$ by Lemma~\ref{lem:abelianrep}. That is,
$\phi_i(t_bt_c^{-1})=I$. Since the commutator subgroup $H_R'$ of $H_R$ is generated normally by $t_bt_c^{-1}$
we get that $\phi_i(f)=I$ for all $f\in H_R'$. It follows that, with respect to some basis of $\C^n$,
\[ \phi (f) =
\left(
   \begin{array}{cc}
      I_{r} & F \\
      0 &  I_{n-r} \\
   \end{array}
 \right)\]
for all $f\in H_R'$.
Since the subgroup of $\GL (n,\C)$ consisting of such matrices is abelian and since $H_R'$ is perfect,
we conclude that $\phi (H_R')$ is trivial. In particular, we have $\phi (t_bt_c^{-1})=I$.
Since $G'=G$ is generated normally by $t_bt_c^{-1}$,
$\phi (G)$ is trivial.

We now fix a subsurface $R$ of genus $g-1$ with one boundary component.
Let $a$ and $b$ be two nonseparating simple closed curves
on $S$ intersecting at one point such that $a\cup b$ is disjoint from $R$.

CASE 1. Suppose that there is a subspace $V$ of dimension $r$ with $2\leq r\leq n-2$ which is a direct sum of eigenspaces
of $L_a$ (Note that there exists such a subspace if $L_a$ has at least three distinct eigenvalues.). Then $V$ is
$\phi (H_R)$--invariant. Hence, $\phi (G)$ is trivial.

CASE 2. Suppose that there is no subspace $V$ as in CASE 1. In particular, $L_a$ has at most two eigenvalues and
each eigenspace of $L_a$ is either $1$--dimensional, or $(n-1)$--dimensional, or $n$--dimensional.
The Jordan form of $L_a$ is one of the following four matrices:\\
 (i) $\lambda I_n$,  (ii)
 $ \left(
   \begin{array}{cccccc}
     \lambda & 1       &  0  & \cdots &  0  & 0 \\
     0       & \lambda & 1 & \cdots & 0  & 0 \\
     \vdots & \vdots & \vdots &\ddots &\vdots &\vdots\\
     0  & 0 & 0 & \cdots &\lambda & 1\\
     0  & 0 & 0 & \cdots & 0      & \lambda\\
   \end{array}
 \right),$
 (iii) $ \left(
   \begin{array}{cccccc}
     \lambda & 0       &  \cdots  & 0 &  0  & 0 \\
     0 & \lambda       &  \cdots  & 0 &  0  & 0 \\
     \vdots & \vdots & \ddots &\vdots &\vdots &\vdots\\
     0   & 0 & \cdots & \lambda & 0  & 0 \\
     0  & 0 & \cdots & 0 &\lambda & 1\\
     0  & 0 & \cdots & 0 & 0      & \lambda\\
   \end{array}
 \right),$\\
 (iv)
 $ \left(
   \begin{array}{cc}
     \lambda I_{n-1} & 0 \\
      0 &  \mu \\
   \end{array}
 \right).$\\
We fix a basis so that the matrix $L_a$ is equal to its Jordan form.

In the case (i), if $x$ is a nonseparating simple closed curve on $S$, then $L_x=\lambda I$
since it is conjugate to $L_a$. Since $G$ is generated by Dehn twists about
nonseparating simple closed curves, $\phi (G)$ is cyclic,
and hence it is trivial.

In the case (ii), the subspace $\ker (L_a-\lambda I)^2$ is a $\phi (H_R)$--invariant subspace of dimension $2$, so that $\phi (G)$ is trivial.

It remain to consider the cases (iii) and (iv). In these cases, the eigenspace $E_\lambda^a$ is of dimension $n-1$.
The eigenspace $E_\lambda^b$ is also of dimension $n-1$. If $E_\lambda^a\neq E_\lambda^b$,
then $E_\lambda^a\cap E_\lambda^b$ is a $\phi (H_R)$--invariant subspace of dimension $n-2(> 1)$. Hence,
$\phi (G)$ is trivial.

Suppose finally that $E_\lambda^a = E_\lambda^b$. It follows from Lemma~\ref{lem:eigenspace}
that $E_\lambda^x = E_\lambda^y$ for any two nonseparating simple closed curves $x$ and $y$ on $S$ intersecting at one point.
We then conclude from Theorem~\ref{thm:dualeqv} that $E_\lambda^a = E_\lambda^x$ for all such nonseparating $x$ on $S$.
Since $\mod (S)$ is generated by Dehn twists about nonseparating
simple closed curves, it follows that for each $f\in G$,
the matrix $\phi (f)$ is upper triangular. Hence, $\phi (G)$ is contained in the subgroup consisting of upper triangular matrices.
But this subgroup of $\GL (n,\C)$ is solvable. Since $G$ is perfect,
we conclude from Lemma~\ref{prop:solv} that $\phi (G)$ is trivial.

This completes the proof of Theorem~\ref{thm:1}.
\qed

\section{Nonorientable surfaces}
The purpose of this section is to prove Theorem~\ref{thm:2}.
So let $\phi:\mod (N)\to \GL (n,\C)$ be a homomorphism, where
$N$ is a nonorientable surface of genus $g\geq 3$ with $p\geq 0$ marked points, and $n\leq g-2$ for odd $g$
and $n\leq g-3$ for even $g$.

\subsubsection{Proof of Theorem~\ref{thm:2}}
If $g=3$ or $g=4$ then $n=1$, and $\GL(1,\C)$ is abelian, so that $\phi $ factors through
the first homology $H_1 (\mod (N);\Z)$.
Since $H_1 (\mod (N);\Z)$ is finite (c.f.~\cite{k98,k02gd}), the result follows. So we assume that $g\geq 5$.

Let $T$ denote the subgroup of the mapping class group $\mod (N)$.
Write $g=2r+1$ if $g$ is odd and $g=2r+2$ if $g$ is even. Hence, we have $r\geq 2$ and $n\leq 2r-1$.

Let $S$ be a compact orientable surface of genus $r$ with one boundary
component. Embed $S$ into $N$, so that the boundary of $S$ bounds a
M\"obius band (resp. Klein bottle with one boundary) on $N$ with $p$ marked points if $g$ is odd (resp. even). Extending diffeomorphisms
$S$ to $N$ by the identity induces a homomorphism
$\eta :\mod (S)\to \mod (N)$. Then the composition
$\phi \eta$ is a homomorphism from $\mod (S)$ to $\GL(n,\C)$.

\[
\begindc{\commdiag}[7]
\obj(10,8)[A]{$\mod (N)$}
\obj(10,2)[B]{$\mod (S)$}
\obj(25,5)[C]{\ $\GL(n,\C)$}
\mor{B}{A}{$\eta$}[1,0]
\mor{B}{C}{$\phi\eta$}[1,0]
\mor{A}{C}{$\phi$}[1,0]
\enddc
\]

If $r\geq 3$ ($g\geq 7$) then $\phi \eta$ is trivial by Theorem~\ref{thm:1}. It follows that $\phi (t_a)=I$ for any Dehn twist
supported on $S$. If $b$ is a two-sided nonseparating simple closed curve on $N$ whose complement in nonorientable,
it follows from Theorems~3.1 and~5.3 in~\cite{k02gd} that a Dehn twist $t_b$ about $b$ is conjugate to
a Dehn twist supported on $S$. Hence, $\phi (t_b)$ is trivial for all such $b$. Since $T$ is generated by such Dehn twists
(c.f.~\cite{k02gd}, proof of Theorem~5.12), we get that $\phi (T)$ is trivial.
We also know that the index of $T$ in $\mod (N)$ is $p!\cdot 2^{p+1}$ (\cite{k02gd}, Corollary~6.2). The conclusion of the theorem now follows.

If $r=2$ ($g=5$ or $g=6$) then $\phi \eta$ is cyclic by Theorem~\ref{thm:1}.
It follows that $\phi (t_at_b^{-1})=I$ for any two nonseparating simple closed curves
on $S$. Let $x$ and $y$ be two two-sided nonseparating simple closed curves on $N$ intersecting at one point
such that the complement of each is nonorientable.
It can easily be shown that there is a diffeomorphism $f$ of $N$ such that
$f(x\cup y)\subset S$, so that $t_xt_y^{-1}$ can be conjugated to $t_at_b^{-1}$,
where $a$ and $b$ are on $S$. Hence, $\phi (t_xt_y^{-1})=I$,
or $\phi (t_x)=\phi (t_y)$. We now apply Theorem~3.1 in~\cite{k02gd} to conclude that
$\phi (t_x)=\phi (t_y)$ for all two-sided nonseparating simple closed curves whose complements are nonorientable.
(Such simple closed curves are called essential in~\cite{k02gd}.)
Since $T$ is generated by such Dehn twists, we get that $\phi(T)$ is cyclic, so that $\phi(T')=\{ I\}$, where $T'$
is the commutator subgroup. Stukow proved that the index of $T'$ in $T$ is $2$ (c.f.~\cite{stuk}, Theorem~8.1).
We conclude that $\phi (\mod (N))$ is a finite group of order at most $p!\cdot 2^{p+2}$.

This finishes the proof of Theorem~\ref{thm:2}.
\qed

\begin{remark}
If $g\geq 7$ and if $N$ is closed, then the above proof shows that
the image of $\phi$ in Theorem~\ref{thm:2} is either trivial or is isomorphic to $\Z_2$.
In fact, it is easy to find a homomorphism whose image is $\Z_2$; send all Dehn twists to the identity and
crosscap slides (=$Y$-homeomorphisms) to an element of order $2$.
\end{remark}
\bigskip

\section{Homomorphisms to $\Aut(F_n)$ and $\Out(F_n)$}
We prove Theorem~\ref{thm:autF-n} in this section. The proof in the
case $n=1$ is trivial. So we assume that $2\leq n\leq 2g-1$.

Let $F_n$ denote the free group of rank $n$.
The action of the automorphism group $\Aut (F_n)$ of $F_n$ on the abelianization of $F_n$ gives rise to a surjective
homomorphism $\eta:\Aut (F_n)\to \GL(n,\Z)$. The kernel of this map is usually denoted by $IA_n$. Let $\Out (F_n)$ denote
the group of outer automorphisms of $F_n$, so that it is the quotient of $\Aut (F_n)$ with the (normal) subgroup
$\Inn (F_n)$ of inner automorphisms.

Suppose that $S$ is a closed surface of genus $2$. Choose five pairwise nonisotopic nonseparating simple closed curves
$c_1,c_2,c_3,c_4,c_5$ on $S$ such that $c_i$ is disjoint from $c_j$ if $|i-j|\geq 2$ and that
$c_i$ intersects $c_j$ at one point if $|i-j|= 1$. Let $t_i$ denote the right Dehn twists about $c_i$.
Then the group $\mod(S)$ is generated by $t_1,t_2,t_3,t_4,t_5$. Let us set
$\sigma=t_1t_2t_3t_4t_5$. It is well--known that $\sigma$ is a
torsion element of order six. One can easily show that
\begin{equation}\label{eqn:genus2}
    \sigma t_i \sigma^{-1}= t_{i+1}
\end{equation}
for each $1\leq i\leq 4$.

\begin{lemma} \label{lem:sigma2}
Let $g=2$. The normal closure of $\sigma^2$ in $\mod(S)$
is equal to the commutator subgroup of $\mod (S)$.
\end{lemma}
\begin{proof}
Let $N$ denote the the normal closure of $\sigma^2$ in $\mod(S)$, the intersection of all
normal subgroups containing $\sigma^2$.
Since any right Dehn twist in $\mod(S)$ about a nonseparating simple closed curve maps
to the generator of $H_1(\mod(S),\Z)$ under the natural homomorphism and since $\sigma^2$
is a product of $10$ such Dehn twists, we see that
$\sigma^2$ is contained in  $[\mod(S),\mod(S)]$. Hence, $N\subset [\mod(S),\mod(S)]$.

Let $q:\mod(S)\to \mod(S)/N$ denote the quotient map. The equality~\eqref{eqn:genus2} implies that
$q(t_1)=q(t_3)=q(t_5)$ and $q(t_2)=q(t_4)$. On the other hand, from the braid relation we get
\begin{eqnarray*}
  q(t_1) q(t_4) q(t_1)
   &=& q(t_1) q(t_2) q(t_1) \\
   &=& q(t_1t_2t_1) \\
   &=& q(t_2t_1t_2) \\
   &=& q(t_2) q(t_1) q(t_2) \\
   &=& q(t_4) q(t_1) q(t_4).
\end{eqnarray*}
Since $t_1$ and $t_4$ commute, we obtain $q(t_1)=q(t_4)$.
It follows that $\mod(S)/N$ is cyclic. In particular, $N$ contains $[\mod(S),\mod(S)]$,
finishing the proof of the lemma.
\end{proof}

\subsubsection{Proof of Theorem~\ref{thm:autF-n}}
Suppose that $H=\Aut (F_n)$. Let $\phi$ be the composition of $\varphi$ with $\eta$, so that
we have a commutative diagram:

\[
\begindc{\commdiag}[8]
\obj(20,8)[M]{$\mod (S)$}
\obj(6,2) [A1]{$1$}
\obj(12,2)[A2]{$IA_n$}
\obj(20,2)[A3]{$\Aut (F_n)$}
\obj(30,2)[A4]{$\GL(n,\Z)$}
\obj(37,2)[A5]{$1.$}
\mor{M}{A2}{$\varphi$}[2,1]
\mor{M}{A3}{$\varphi$}[1,0]
\mor{M}{A4}{$\phi$}[1,0]
\mor{A1}{A2}{$ $}[1,0]
\mor{A2}{A3}{$ $}[1,0]
\mor{A3}{A4}{$\eta $}[1,0]
\mor{A4}{A5}{$ $}[1,0]
\enddc
\]

We now apply Theorem~\ref{thm:1}. If $g\geq 3$ then $\phi$ is trivial, implying
that the image of $\varphi$ is contained in $IA_n$. Since $IA_n$ is torsion--free by a result of
Baumslag-Taylor~\cite{bt}, all torsion elements in $\mod (S)$ are contained
in the kernel of $\varphi$.
Since $\mod(S)$ is is generated by torsion elements (cf.~\cite{m-pap,h-k,k05tams}),
we conclude that $\varphi$ is trivial.

Suppose now that $g=2$. Theorem~\ref{thm:1} gives us that the image of $\phi$ is cyclic,
implying that the commutator subgroup $N$ of $\mod(S)$ is contained in the kernel of $\phi$.
Therefore, $\varphi (N)$ is contained in $IA_n$. Since $IA_n$ is torsion--free and since $\sigma^2\in N$
is a torsion element (of order $3$), we have $\varphi(\sigma^2)=1$.
It follows from Lemma~\ref{lem:sigma2} that $\varphi(N)$ is trivial.
Consequently, $\varphi$ factors through the natural homomorphism $\mod(S)\to H_1(\mod(S),\Z)$, and
so the image of $\varphi$ is a quotient of $\Z_{10}$.

This completes the proof of Theorem~\ref{thm:autF-n} for $\Aut (F_n)$.

The case $H=\Out(F_n)$ is completely similar
and uses the fact that the subgroup $IA_n/\Inn(F_n)$ is torsion--free
(c.f.~\cite{bt}). \qed

\end{document}